\documentclass[11pt]{amsart}
\usepackage{mathtools}
\usepackage{amssymb}
\usepackage{amsmath}
\usepackage{amsthm}
\usepackage[margin=1in]{geometry}
\usepackage[utf8]{inputenc}
\newtheorem{theorem}{Theorem}
\newtheorem{corollary}[theorem]{Corollary}

\newtheorem{proposition}[theorem]{Proposition}
\theoremstyle{definition}

\theoremstyle{remark}

\newcommand{\RR}{\mathbb{R}}
\newcommand{\ZZ}{\mathbb{Z}}

\DeclareMathOperator{\ad}{ad} %
\DeclareMathOperator{\adstar}{ad^{\star}} %

\newcommand{\llangle}{\langle\!\langle}
\newcommand{\rrangle}{\rangle\!\rangle}

\newtheorem*{acknowledgement}{Acknowledgement}

\mathtoolsset{showonlyrefs}

\begin{document}

\title{Nonpositive Curvature of the quantomorphism group and quasigeostrophic motion}
\author[J. Lee]{Jae Min Lee}
\address{Department of Mathematics, KTH Royal Institute of Technology, 100 44 Stockholm, Sweden}
\email{lee9@kth.se}

\author[S.C. Preston]{Stephen C. Preston}
\address{Department of Mathematics, Brooklyn College and the Graduate Center, City University of New York, NY 11106, USA}
\email{stephen.preston@brooklyn.cuny.edu}

\date{\today}

\maketitle

\begin{abstract}
In this paper, we compute the sectional curvature of the quantomorphism group $\mathcal{D}_q(M)$ whose geodesic equation is the quasi-geostrophic (QG) equation in geophysics and oceanography, for flows with a stream function depending on only one variable. Using this explicit formula, we will also derive a criterion for the curvature operator to be nonpositive and discuss the role of the Froude number and the Rossby number on curvature. The main technique to obtain a usable formula is a simplification of Arnold's general formula in the case where a vector field is close to a Killing field, and then use the Green's function explicitly. We show that nonzero Froude number and Rossby numbers both tend to stabilize flows in the Lagrangian sense.
\end{abstract}

\section{Introduction}

There are two classical viewpoints on the motion of a fluid. First, the Eulerian perspective concerns $u(t,x)$, the velocity of a fluid particle located at the point $x$ at time $t$, and one studies the evolution equation of $u$ with the prescribed initial/boundary conditions. In the Lagrangian formalism, one considers the function $\eta(t,x)$, which is the position at time $t$ of a fluid particle which at time zero was at $x$. So one can think of the collection of $\eta(t,\cdot)$ as giving the configuration of the particles at each time $t$ and can recover the Eulerian description via $u(t,\cdot)=\eta_t \circ \eta^{-1}$. In the case of ideal fluid on a Riemannian manifold $M$, the configuration space is $\mathcal{D}_\mu(M)$, the group of volume preserving diffeomorphisms on $M$ where $\mu$ is the volume form on $M$. In his beautiful paper in 1966, Arnold \cite{A1966} observed that the Euler equation for ideal fluid can be realized as the geodesic equation on $\mathcal{D}_\mu(M)$ endowed with the right-invariant kinetic energy metric, and this observation was rigorously justified by Ebin and Marsden in 1970 \cite{EM1970}. Since then, geodesic equations on the diffeomorphism groups endowed with an invariant metric have been studied extensively. Invariance leads to a reduction of order to a first-order equation on the Lie algebra, which is called the Euler-Arnold equation.

The quasi-geostrophic equation (QG) describes large scale flows in atmosphere and ocean which have large horizontal to vertical aspect ratio. Here, quasi-geostrophy means that Coriolis force and horizontal pressure gradient forces are nearly in balance, which allows the momentum equation for the flow to be prognostic and include nonlinear dynamics. In terms of the stream function $\psi(t,x,y)$ of the velocity $u$ of the barotropic fluid, the QG equation in the $\beta$-plane approximation is given by
\begin{equation}\label{QGb}
\partial_t\left(\Delta \psi-\alpha^2 \psi\right)+\{\psi, \Delta \psi\}+\beta \psi_x=0,
\end{equation}
where $\alpha^2$ denotes the Froude number and $\beta$ is the Rossby number, the gradient for the Coriolis parameter. Here, $\{\cdot,\cdot\}$ is the Poisson bracket, i.e., $\{g,h\}=h_y g_x-g_x h_y$. The Coriolis parameter $f$ is approximated in the $\beta$-plane by $f=f_0+\beta y$ with constants $f_0$ and $\beta$. The case when $\beta=0$ is the $f$-plane approximation. The Froude number $\alpha^2$ is a nondimensionalized parameter defined by
$$\alpha:=\frac{u_0}{\sqrt{g_0 l_0}},$$
where $u_0$ is the velocity scale, $g_0$ is the gravitational constant, and $l_0$ is the horizontal length scale. So $\alpha$ measures the effect of gravity and $\alpha \ll 1$ in the mesoscale motions of the atmosphere and oceans in the midlatitudes. Additionally for $\alpha$ and $\beta$ both nonzero, equation \eqref{QGb} is the Hasegawa-Mima equation arising in plasma dynamics~\cite{ZP1994}.  The equation \eqref{QGb} can also be written in terms of the potential vorticity as
\begin{equation}\label{streamformulation}
\partial_t\omega+\{\psi,\omega\}=0,\;\;\;\;\;\omega=\Delta \psi-\alpha^2 \psi+\beta y,
\end{equation}
which is similar to the vorticity-stream formulation of the 2-dimensional incompressible Euler equation. The QG equation can be derived as the inviscid limit of the rotating shallow-water equations, as well. For more mathematical theory of atmospheric and oceanic fluid, see Majda \cite{M2003}. For more comprehensive background on the geostrophical fluid dynamics, see Pedlosky \cite{Ped2013}. It is important to note that equation \eqref{streamformulation} is \emph{not} the ``surface quasi-geostrophic'' (SQG) equation; the SQG equation is when $\omega = \sqrt{-\Delta} \psi$, and it has completely different properties. See \cite{BHP2019} and \cite{W2016} for the geometric approach to SQG, and references therein for other aspects.

From the geometric point of view, the QG equation is of interest since it is an example of the Euler-Arnold equation. In 1994, Zeitlin-Pasmanter \cite{ZP1994} showed that the QG equation can arise as the Euler-Arnold equation in the infinite dimensional Lie algebra and its central extension, without constructing the full group. They also computed the sectional curvature and showed that it is negative in the section spanned by the cosinusoidal stationary flows. In 1998, Holm-Zeitlin \cite{HZ1998} showed that the QG equation in the $f$- and $\beta$-plane approximations are the geodesic equations on the group of symplectic diffeomorphisms by using variational principles for QG dynamics.  Also, in 2008, Vizman \cite{V2008} showed that the equation \eqref{QGb} is the Euler-Arnold equation on the central extension of the group of Hamiltonian diffeomorphisms in the case when $\alpha=0$. Finally, Ebin-Preston \cite{EP2015} showed in 2015 that the QG equation is the geodesic equation on a central extension of the quantomorphism group (thus constructing the group corresponding to the Lie algebra in \cite{ZP1994}). 

On a contact manifold $(M, \theta)$, the quantomorphism group $\mathcal{D}_q(M)$ is defined as the space of diffeomorphisms on $M$ that preserve the contact form $\theta$ exactly. So the quantomorphisms group is a subgroup of the contactomorphism group $\mathcal{D}_\theta(M)$, whose elements preserve the contact structure, i.e., $\eta^*\theta = e^{\lambda}\theta$ for some $\lambda\colon M\to\mathbb{R}$. If the contact form is regular, then $\mathcal{D}_\theta(M)$ is related to a symplectic manifold by a Boothby-Wang fibration and the tangent space of $\mathcal{D}_q(M)$ can be identified with the space of functions $f\colon M \to \mathbb{R}$ such that $E(f)=0$, where $E$ is the Reeb field. Furthermore, one can show that $\mathcal{D}_q(M) \subset \mathcal{D}_\theta(M)$ is a totally geodesic submanifold. For more Riemannian geometry of the contactomorphism group in general, see Ebin-Preston \cite{EP2015}.

As in the finite dimensional Lie group case, the sectional curvature of the diffeomorphism group provides information about the stability of geodesics, which we call the Lagrangian stability. For example, positive curvature in all sections implies that geodesics with close initial data locally converge (stability) while negative sectional curvature implies that the geodesics spread apart (instability). Eulerian and Lagrangian stability are different but related: for example if a fluid is stable in the Eulerian sense, then the linearized Lagrangian perturbations can grow at most polynomially in time; see the second author's paper~\cite{P2004}. For more discussions on the curvature of the Euler-Arnold equations in general, see Khesin et al. \cite{KLMP2013}.

In this paper, we compute the sectional curvature $K(X,Y)$ of the quantomorphism group $\mathcal{D}_q(M)$ by the plane spanned by $X,Y \in T_{\mathrm{Id}}\mathcal{D}_q(M)$. Then from the explicit formula of the curvature, we will find a necessary and sufficient condition for the curvature operator $R_X:Y \mapsto K(X,Y)$ to be nonpositive. The explicit computation of the curvature formula is inspired by the work of the second author \cite{P2005} where the nonpositive curvature criterion for the area-preserving diffeomorphism group of a rotationally symmetric surface was presented. A similar computation was done for incompressible axisymmetric fluids in \cite{WP2017}.

The outline of the paper is following. In Section 2, we will review the Riemannian geometry of the quantomorphism group and sectional curvature formula. We will observe that the curvature formula simplifies significantly when one of the tangent vector is chosen to be a function of only the $y$-variable. Then in Section 3, we will compute the sectional curvature formula explicitly by using the Green's function directly and writing the curvature formula in terms of the combination of first integrals of known quantities. Then we will derive the nonpositive curvature criterion and discuss the role of the Froude number and the Rossby number on the curvature. Finally, Section 4 contains some conclusions and remarks.


\section{Riemannian geometry of quantomorphism group}

\subsection{The space of quantomorphisms}

Let $N$ be a 2-dimensional manifold with symplectic form $\omega$ (a nowhere-zero $2$-form). On top of $N$, there is a 3-dimensional manifold $M$ with a contact form $\theta$ such that $\theta\wedge d\theta$ is nowhere-zero, and a projection map
$\pi\colon M \to N$ satisfying $\pi^\ast \omega=d\theta$. Recall that for the contact form $\theta$, there is a unique vector field $E$, called the Reeb field, satisfying the two conditions $\theta(E)=1$ and $\iota_E d\theta=0$. Here for simplicity we will assume $N$ is the flat cylinder $N=\mathbb{S}^1\times [0,L]$ with $M = \mathbb{S}^1\times [0,L] \times\mathbb{S}^1$, where $\mathbb{S}^1 = \RR/2\pi\ZZ$, with $\theta = dz - y\,dx$ and $\omega = dx\wedge dy$. In this case, the Reeb field is $E=\partial_z$.

The space of quantomorphisms $\mathcal{D}_q(M)$ consists of diffeomorphisms $\eta$ on $M$ that preserve the contact form exactly, i.e., $\eta^\ast \theta=\theta$. Its tangent space at the identity consists of vector fields $X$ such that $\mathcal{L}_X \theta=0$, and such a vector field $X$ is uniquely determined by the function $\psi=\theta(X)$ via the formula $\iota_Xd\theta + d\psi=0$, and we write $X = S_{\theta}\psi$, following \cite{EP2015}. In the present case with $\theta = dz-y\,dx$, we have
\begin{equation}\label{Sthetadef}
X = S_{\theta}\psi = -\psi_y\,\partial_x + \psi_x \,\partial_y + (\psi + y\psi_y) \,\partial_z.
\end{equation}
This preserves the contact form iff $\psi_z=0$, and conversely any such function with $\psi_z=0$ gives a quantomorphism vector field.
That is, we can identify elements $X \in T_\mathrm{Id}\mathcal{D}_q(M)$ with $E$-invariant functions on $M$, which are identified with all functions on $N$.

\subsection{The Riemannian structure of $\mathcal{D}_q(M)$}
With the identification mentioned above, on the space of quantomorphisms $\mathcal{D}_q(M)$, we put a right-invariant metric which at the identity is given by
\begin{equation}
\llangle X, Y\rrangle=\int_N \alpha^2 \psi g+\langle \nabla \psi,\nabla g\rangle d\nu,\;\;\;\;X,Y \in T_\mathrm{Id}\mathcal{D}_q(M),
\end{equation}
for $X = S_{\theta}\psi$ and $Y = S_{\theta}g$,
where $\nu$ is the volume form on $N$, for a parameter $\alpha^2$ representing the Froude number.\footnote{In some geometries, such as on $S^3$ with its standard contact form, this metric is the $L^2$ metric on $X$; here the Euclidean metric is not compatible with the contact structure, and we prefer to use the Euclidean metric to obtain simpler equations.} The Lie algebra structure on $\mathcal{D}_q(M)$ is given by
\begin{equation}
\ad_X Y=-[S_{\theta}\psi,S_{\theta}g]=-S_{\theta}\{\psi,g\},
\end{equation}
where $X=S_\theta \psi$ and $Y=S_\theta g$ and $\{\psi,g\} = \psi_x g_y - \psi_y g_x$ is the usual Poisson bracket on $N$.
Then the Euler-Arnold equation on $\mathcal{D}_q(M)$ is
\begin{equation}
\left(\Delta-\alpha^2\right)\psi_t+\left\{\Delta \psi,\psi\right\}=0,
\end{equation}
which is the quasi-geostrophic equation in $f$-plane approximation on $N$; see \cite{EP2015}.

Now, we consider the central extension of $T_\mathrm{Id}\mathcal{D}_q(M)$ by $\mathbb{R}$ which is given by a cocycle of the form
$$b(\psi,g)=\int_N \chi\left\{\psi,g\right\}d\nu,$$
where $\chi$ is a fixed function, in this case given by $\chi(x,y)=y$. Then the new Lie algebra on $C^\infty(M) \times \mathbb{R}$ is given by
\begin{equation}
\ad_{\tilde{X}} \tilde{Y}=-\left[(\psi,\beta),(g,\gamma)\right]=-\left(S_{\theta}\{\psi,g\},b(\psi,g)\right),
\end{equation}
and the right-invariant metric is given at the identity by
\begin{equation}
\llangle \tilde{X},\tilde{Y} \rrangle=\llangle (\psi,\beta),(g,\gamma) \rrangle=\int_N \big(\alpha^2 \psi g+\langle \nabla \psi,\nabla g\rangle \big)\, d\nu+\beta \gamma,
\end{equation}
where $\tilde{X}=(\psi,\beta)$ and $\tilde{Y}=(g,\gamma)$. Then we can compute that
\begin{equation}
\adstar_{\tilde{X}} \tilde{Y}= \Big(S_{\theta} (\alpha^2 - \Delta)^{-1}\big( \alpha^2 \{\psi,g\} - \{\psi,\Delta g\} - \beta \{\psi,\chi\}\big), 0\Big),
\end{equation}
and the corresponding geodesic equation is
\begin{equation}
\beta_t=0,\;\;\;\;\left(\Delta-\alpha^2\right)\psi_t+\{\psi,\Delta \psi\}+\beta \psi_x=0,
\end{equation}
which is the equation \eqref{QGb}, the quasi-geostrophic equation in $\beta$-plane approximation. We can also write this equation in terms of the potential vorticity $\omega$ as following:
\begin{equation}
\omega_t+\{\psi,\omega\}=0,\;\;\;\;\omega=\psi_{xx}+\psi_{yy}-\alpha^2 \psi +\beta y.
\end{equation}

If we assume that $\psi$ is a function of only the $y$-variable, then $\psi$ is a steady solution since $\omega=\psi''(y)-\alpha^2 \psi(y)+\beta y$ and $\{\psi,\omega\}=0$.

\subsection{Sectional curvature formula}

Recall that Arnold's sectional curvature formula is
\begin{align}\label{arnold}
K(X,Y):=&\langle R(X,Y)Y,X\rangle\\
= &\tfrac{1}{4} \Big( \lvert \adstar_X Y+\adstar_Y X \rvert^2 +2 \langle \ad_X Y, \adstar_Y X - \adstar_X Y\rangle\\
&\hspace{1.3in}-3\lvert \ad_X Y\rvert^2 - 4\langle \adstar_X X, \adstar_Y Y\rangle\Big),
\end{align}
where $X,Y \in T_{\mathrm{Id}}\mathcal{D}_q(M)$ which are identified by $X=(\psi,\beta)$ and $Y=(g,\gamma)$ for functions $\psi,g:N \to \mathbb{R}$ and $\beta,\gamma \in \mathbb{R}$. From the assumption that $\psi$ is a function of only the $y$-variable, we can simplify the formula \eqref{arnold} in a nice form so that we can use the explicit computation technique suggested by the second author~\cite{P2005}.

Observe that
$$\adstar_X Y=\left(-\Lambda^{-1}(\psi'\Lambda g_x),0\right)$$
where $\Lambda=\alpha^2-\Delta$, and
$$-\ad_X Y=(-\psi'g_x,0),$$
which is very close to $\adstar_X Y$. In fact, these two are exactly the same when $\psi'' \equiv 0$. Define the following nonsymmetric commutator operator
$$D(X,Y):=\adstar_X Y+\ad_X Y.$$
Note that by right-invariance, the deformation tensor of $X$ is given by
\begin{equation*}
\mathrm{Def}\;X(Y,W)=\llangle \nabla_YX,W\rrangle + \llangle Y, \nabla_WX\rrangle = \llangle \ad_X Y+\adstar_X Y,W \rrangle.
\end{equation*}
Hence, we can conclude that the operator $D(X,Y):=\mathrm{Def}\;X(Y)$ satisfies the condition that $D(X,\cdot)=0$ if and only if $X$ is an isometry. For example, $\psi''(y)=0$ implies that $X$ is an isometry. So, in terms of this operator $D$, we can write
\begin{equation}\label{defor}
\adstar_X Y=-\ad_X Y+D(X,Y),
\end{equation}
and we have the following simplification of the Arnold curvature formula in the case when $D(X,Y)$ is simple.

\begin{proposition}
The Arnold curvature formula can be written in terms of the operator $D$ as following:
\begin{equation}\label{2termcurvature}
K(X,Y)=\frac{1}{4}\left\vert \adstar_Y X +D(X,Y) \right\vert^2-\llangle \ad_X Y,D(X,Y)\rrangle - \llangle D(X,X), D(Y,Y)\rrangle.
\end{equation}
\end{proposition}

\begin{proof}
By substituting the equation \eqref{defor} and expanding, we get
\begin{align*}
\langle R(X,Y)Y,X \rangle=&\frac{1}{4}\lvert \adstar_Y X\!-\!\ad_X Y\!+\!D(X,Y)\rvert^2\!+\frac{1}{2}\langle \ad_X Y, \adstar_Y X\!+\!\ad_X Y-D(X,Y)\rangle\! \\
&\qquad\qquad -\frac{3}{4}\lvert \ad_X Y\rvert^2 - \langle D(X,X), D(Y,Y)\rangle \\
=&\frac{1}{4}\lvert \adstar_Y X\rvert^2+\frac{1}{4}\lvert \ad_X Y \rvert^2+\frac{1}{4}\lvert D(X,Y) \rvert^2-\frac{1}{2}\langle \adstar_Y X,\ad_X Y\rangle  \\
&\;\;\;\;- \langle D(X,X), D(Y,Y)\rangle+\frac{1}{2}\langle \adstar_Y X,D(X,Y)\rangle-\frac{1}{2}\langle \ad_X Y,D(X,Y)\rangle\\
&\;\;\;\;+\frac{1}{2}\langle \ad_X Y,\adstar_Y X \rangle+\frac{1}{2}\lvert \ad_X Y \rvert^2-\frac{1}{2}\langle \ad_X Y, D(X,Y) \rangle-\frac{3}{4}\lvert \ad_X Y \rvert^2\\
=&\frac{1}{4}\lvert \adstar_Y X+D(X,Y) \rvert^2-\langle \ad_X Y, D(X,Y) \rangle - \langle D(X,X), D(Y,Y)\rangle.
\end{align*}
\end{proof}

Note that if $X$ is a steady solution of the Euler equation, then $\adstar_XX=D(X,X)=0$, and the last term disappears.

We will compute the sectional curvature in the case when $\psi=\psi(y)$ generates a steady solution, and in this case,
\begin{align*}
\adstar_Y X+D(X,Y)=&\left(\Lambda^{-1}\left(-2\partial_y(\psi''g_x)+(\alpha^2 \psi'-\beta)g_x\right),0\right),\\
\llangle \ad_X Y,D(X,Y) \rrangle=&\int_N (\psi''g_x)^2 d\nu,
\end{align*}
and finally the curvature formula becomes
\begin{equation}\label{curvature}
K(X,Y)=\int_N \left(\partial_y(\psi''g_x)-\frac{1}{2}(\alpha^2 \psi'-\beta)g_x\right)\Lambda^{-1}\left(\partial_y(\psi''g_x)-\frac{1}{2}(\alpha^2 \psi'-\beta)g_x\right) d\nu-\int_N (\psi''g_x)^2 d\nu.
\end{equation}
We can see that if $\alpha=\beta=0$, then the curvature formula reduces to
\begin{align*}
K(X,Y)&=\int_N \left(\partial_y(\psi''g_x)\right)(\partial_x^2+\partial_y^2)^{-1}\left(\partial_y(\psi''g_x)\right)\, d\nu-\int_N (\psi''g_x)^2 \,d\nu \\
&= \int_N (\psi''g_x)_x (\partial_x^2+\partial_y^2)^{-1} (\psi''g_x)_x \, d\nu \le 0,
\end{align*}
which reproduces the nonpositive sectional curvature of the 2-dimensional area preserving diffeomorphism group case from \cite{P2005}.

\section{Explicit curvature formula and Nonpositive criterion}

\subsection{Green's function}

To proceed with the explicit computation of the formula \eqref{curvature}, we compute the Greens function for $\Lambda^{-1}$ explicitly. We will expand the function $g$ in terms of the Fourier series in $x$ variables of the form
$$g(x,y)=\sum_{n \in \mathbb{Z}}g_n(y)e^{inx}.$$
In our domain, the stream function must be \emph{constant} on the boundary segments $y=0$ and $y=L$, and thus when $n\ne 0$ we must have $g_n(0)=g_n(L)=0$.
Then the boundary value problem associated with the Green's function for $\Lambda^{-1}$ reduces to an ODE for the functions in $y$ variables. We obtain the following BVP for $u(y) = G(y,s)$:
$$\left\{
  \begin{array}{l l}
      -u''(y)+\lambda^2 u(y)=\delta(y-s)\\
      u(0)=0=u(L)
   \end{array} \right.$$
whose explicit solution is given by
\begin{equation}\label{greenfunction}
G(y,s) = \frac{1}{\lambda \sinh{L\lambda}} \begin{cases}
\sinh{\lambda(L -s)} \sinh{\lambda y} & 0\le y\le s, \\
\sinh{\lambda(L -y)} \sinh{\lambda s} & s\le y\le L.
\end{cases}
\end{equation}
where $\lambda^2=\alpha^2+n^2$.

\subsection{Explicit computation of the curvature}

Now, we want to compute the sectional curvature formula \eqref{curvature} explicitly using the Green's function. By first substituting the Fourier expansion of $g$, we have
\begin{equation}\label{fouriersum}
K(X,Y) = \sum_{n\in\ZZ} n^2 K_n,
\end{equation}
where
\begin{equation}\label{curv2}
K_n = \frac{1}{4} \int_0^{L}\overline{\phi_n(y)}\Big(\lambda^2-\frac{d^2}{dy^2}\Big) \phi_n(y) \, dy-\int_0^{L} q(y)^2 \lvert g_n(y)\rvert^2 \, dy,
\end{equation}
where
\begin{equation}\label{pqdef}
p(y) = \alpha^2 \psi'(y) - \beta, \qquad q(y) = \psi''(y),
\end{equation}
and
\begin{equation}\label{psidef}
\lambda^2 \phi_n(y) - \phi_n''(y) = p(y)g_n(y) - 2 \, \frac{d}{dy} \big(q(y) g_n(y)\big), \qquad \phi_n(0)=\phi_n(L)=0.
\end{equation}

\begin{proposition}\label{Kncompyprop}
For any function $g_n\colon [0,L]\to \mathbb{C}$, the $n^{\text{th}}$ term $K_n$ in the curvature \eqref{curv2} is given by
\begin{equation}\label{curv3}
K_n = \frac{2}{\lambda \sinh{L \lambda}} \int_0^L \int_0^y \xi(y) \eta(z) \mathrm{Re}{\big( \overline{g_n(z)} g_n(y) \big)} \, dz \, dy,
\end{equation}
where $\lambda^2 = \alpha^2 + n^2$, $p(y) = \alpha^2 \psi'(y) - \beta$, $q(y) = \psi''(y)$, and
\begin{align}
\eta(y) &= \tfrac{1}{2} p(y) \sinh{\lambda y} + \lambda q(y) \cosh{\lambda y} \label{etadef} \\
\xi(y) &= \tfrac{1}{2} p(y) \sinh{\lambda(L-y)} - \lambda q(y) \cosh{\lambda(L-y)}. \label{xidef}
\end{align}
\end{proposition}

\begin{proof}
From the formula \eqref{greenfunction} and \eqref{psidef}, we have
\begin{align*}
\phi_n(y) &= \frac{1}{\lambda \sinh{\lambda L}} \int_0^y \sinh{\lambda z} \sinh{\lambda (L-y)} \Big( p(z)g_n(z) - 2\tfrac{d}{dz} \big( q(z) g_n(z)\big) \Big)\, dz \\
&\qquad\qquad + \frac{1}{\lambda \sinh{\lambda L}} \int_y^L \sinh{\lambda y} \sinh{\lambda (L-z)} \Big( p(z)g_n(z) - 2\tfrac{d}{dz} \big(q(z) g_n(z)\big)\Big) \, dz.
\end{align*}
Integrate by parts to remove the derivative on $\frac{d}{dz}\big( q(z)g_n(z)\big)$, and we obtain (after vanishing of the boundary term)
\begin{equation}\label{psisimplified}
\begin{split}
\phi_n(y) &= \frac{1}{\lambda \sinh{\lambda L}} \int_0^y \sinh{\lambda(L-y)} \Big( \sinh{\lambda z}  p(z) + 2 \lambda q(z) \cosh{\lambda z}\Big) g_n(z) \, dz \\
&\qquad\qquad + \frac{1}{\lambda \sinh{\lambda L}} \int_y^L \sinh{\lambda y} \Big( \sinh{\lambda(L-z)} p(z) - 2 \lambda q(z) \cosh{\lambda(L-z)} \Big) g_n(z) \, dz \\
&= \frac{2}{\lambda \sinh{\lambda L}} \int_0^y \sinh{\lambda(L-y)} \, \eta(z) g_n(z) \, dz + \frac{2}{\lambda \sinh{\lambda L}} \int_y^L \sinh{\lambda y} \, \xi(z) g_n(z)\, dz.
\end{split}
\end{equation}
The derivative is easily computed to be
\begin{multline}\label{psiderivative}
\phi_n'(y) = 2q(y) g_n(y) - \frac{2}{\sinh{\lambda L}} \int_0^y \cosh{\lambda(L-y)} \,\eta(z) g_n(z) \, dz \\
 + \frac{2}{\sinh{\lambda L}} \int_y^L \cosh{\lambda y} \,\xi(z)  g_n(z)\,dz.
\end{multline}

Now using \eqref{psidef} in \eqref{curv2} and integrating by parts we get
$$
K_n = \frac{1}{4} \int_0^{L}\overline{\phi_n(y)} p(y)g_n(y) \, dy + \frac{1}{2} \int_0^L \overline{\phi_n'(y)} q(y) g_n(y) \, dy
-\int_0^{L} q(y)^2 \lvert g_n(y)\rvert^2 \, dy.
$$
Inserting the expressions \eqref{psisimplified} and \eqref{psiderivative} into this, and recalling the definitions of $\eta$ and $\xi$ from \eqref{etadef}--\eqref{xidef}, we obtain
$$ K_n 
= \frac{1}{\lambda \sinh{\lambda L}} \int_0^L \int_0^y \xi(y) \eta(z) \overline{g_n(z)} g_n(y) \, dz\, dy  + \frac{1}{\lambda \sinh{\lambda L}}  \int_0^L \int_y^L \eta(y) \xi(z) \overline{g_n(z)} g_n(y) \, dz \, dy.$$
Finally interchanging the order of integration and switching $y$ and $z$ in the second integral turns this into \eqref{curv3}.
\end{proof}
%

To proceed further, we observe the following fact.

\begin{theorem}\label{generalintegral}
Suppose $\eta,\xi\colon [0,L]\to \mathbb{R}$ are given functions, and that the function $R(y)=\xi(y)/\eta(y)$ is meromorphic. Then the bilinear form
\begin{equation}\label{bilinearform}
g\mapsto B(g,g) := 2\int_0^L \int_0^y  \xi(y) \eta(z) \mathrm{Re} \big(\overline{g(z)} g(y)\big) \,dz \, dy
\end{equation}
is nonpositive for all $g\colon [0,L]\to \mathbb{C}$ if and only if $R$ is nowhere zero or infinite, and the function $R$ is increasing and nonpositive on $[0,L]$.
\end{theorem}

\begin{proof}
Suppose $R(y)$ is well-defined on $[0,L]$. Let $H(y) = \int_0^y \eta(z) g(z)\,dz$. Then we have
\begin{equation}\label{bilinearg}
\begin{split}
B(g,g) &= \int_0^L R(y) \overline{H'(y)} H(y) + R(y) H'(y) \overline{H(y)} \, dy \\
&= \int_0^L R(y) \frac{d}{dy} \lvert H(y)\rvert^2 \, dy \\
&= R(L) H(L)^2 - \int_0^L R'(y) \lvert H(y)\rvert^2 \, dy.
\end{split}
\end{equation}
If $R(L)\le 0$ and $R'(y)\ge 0$, then $B(g,g)\le 0$ for every $g$.\\
Conversely, suppose that $B(g,g) \le 0$ for every $g$. We first claim that the function $R$ cannot have any singularity in $[0,L]$. If $R$ is singular at some point $y_0$, then $\zeta = 1/R = \eta/\xi$ has a zero at this point since $R$ is meromorphic. Consider only functions $g$ with support in a small neighborhood $U=(a,b)$ of $y_0$ where $\zeta$ has no zeroes in $[a,b]$ other than $y_0$. The integral \eqref{bilinearform} can be rewritten after a change of integration order as
$$
B(g,g) = 2\int_0^L \int_z^L \xi(y) \eta(z) \mathrm{Re}{\big( \overline{g(z)} g(y)\big)} \, dy \, dz
= 2\int_a^b \int_z^b \xi(y)\eta(z) \mathrm{Re}{\big(\overline{g(z)} g(y)\big)}  \,dy \, dz.$$
Now set $J(z) = \int_z^b \xi(y) g(y)\,dz$. Then
$$ B(g,g) = -\int_a^b \zeta(z) \,\frac{d}{dz}\lvert J(z)\rvert^2 \, dz
= \zeta(a) \lvert J(a)\rvert^2 + \int_a^b \zeta'(z) \lvert J(z)\rvert^2 \, dz.  $$
We know $\zeta(y_0)=0$ for a unique $y_0\in (a,b)$, and we consider the sign of $\zeta(a)$, since by assumption $\zeta(a)\ne 0$.
\begin{itemize}
\item If $\zeta(a)>0$ then we may clearly choose $g$ so that $\lvert J(a)\rvert$ is large compared to $\lVert J\rVert_{L^2(a,b)}$ and obtain positivity of $B(g,g)$.
\item If $\zeta(a)<0$ then $\zeta'$ must be positive at some  $c_0\in (a,y_0)$, and we may choose $g$ so that $J$ is supported in a small neighborhood of $c_0$ and again obtain positivity of $B(g,g)$.
\end{itemize}
Thus if $B(g,g)\le 0$ for every $g$, then $R$ cannot have a pole. By using a similar argument, we can show that the function $R$ cannot have a zero.

Lastly, we claim that the function $R$ is increasing and nonpositive on $[0,L]$.
If there is any point $y_0 \in (0,L)$ with $R'(y_0)<0$, in a small neighborhood of $y_0$ we can choose $H$ nonzero in this neighborhood and zero outside, and obtain a contradiction in the nonpostivity of \eqref{bilinearg}. Hence we must have $R'(y)\ge 0$ everywhere in $[0,L]$ by continuity. Meanwhile if $R(L)>0$, then we can choose $H$ such that $\lvert H(L)\rvert$ is large but $\lVert H\rVert_{L^2}$ is small on $[0,L]$, and again obtain a contradiction. This completes the proof of the converse.
\end{proof}

%
%
%
%
%

We now apply Theorem \ref{generalintegral} to the formula \eqref{curv3}.

\begin{proposition}\label{curvaturetheorem}
Suppose $f$ is analytic and $p(y) = \alpha^2 \psi'(y) - \beta$ and $q(y)=\psi''(y)$. For $n\in \mathbb{N}$, the $n^{\text{th}}$ term $K_n$ in the sectional curvature given by \eqref{curv3} is nonnegative for all $g_n\colon [0,L]\to \mathbb{C}$ if and only if $\eta$ and $\xi$ given by \eqref{etadef}--\eqref{xidef} have no isolated zeroes in $(0,L)$ and
\begin{align}
\alpha^2 p(y)^2 + 2 \alpha^2 p(y) p''(y) - (6 \alpha^2 + 4n^2) p'(y)^2 &\le 0  \label{pcritDE} \\
\frac{2\lambda p'(L)}{\alpha^2 p(L) \sinh{(\lambda L)} + 2\lambda \cosh{(\lambda L)} \, p'(L)} &\ge 0.\label{pcritBC}
\end{align}
\end{proposition}

\begin{proof}
Observe that $q(y) = \frac{1}{\alpha^2} p'(y)$ in \eqref{pqdef}. With $R = \xi/\eta$ and $\eta$ and $\xi$ given by \eqref{etadef}--\eqref{xidef}, it is easy to compute that $R'(y)\ge 0$ translates into the differential inequality \eqref{pcritDE}, while $R(L)\le 0$ translates into the boundary condition \eqref{pcritBC}. Thus Theorem \ref{generalintegral} yields the conclusion.
\end{proof}

Note that the condition for sign-definiteness of the curvature is that \eqref{pcritDE} for \emph{all} nonzero integers $n$; so far we have only considered one integer at a time. 

\begin{corollary}\label{nonnegativecurvaturecor}
Suppose $X = f(y)\,\partial_x$ is a steady shear-flow solution to the quasigeostrophic equation. Then
the curvature operator $Y\mapsto R(Y,X)X$ is nonnegative iff
\begin{equation}\label{W2ndorderbest}
\alpha^2 p(y)^2 + 2 \alpha^2 p(y) p''(y) - (6 \alpha^2 + 4) p'(y)^2 \le 0 \\
\end{equation}
for all $y\in [0,L]$, and
\begin{equation}\label{W1storderbest}
\alpha^2 p(L)p'(L) + 2p'(L)^2 \sqrt{\alpha^2+1} \coth{\big(L\sqrt{\alpha^2+1}\big)} \ge 0,
\end{equation}
where $p(y) = \alpha^2 \psi'(y) - \beta$.
\end{corollary}

\begin{proof}
The worst-case scenario in \eqref{pcritDE}
is the smallest value of $\lambda=\sqrt{\alpha^2+n^2}$, which is
when $n=1$ (since $n=0$ does not show up in the formula \eqref{fouriersum}). For this $n$ we obtain \eqref{W2ndorderbest}.

On the other hand, the boundary condition from \eqref{pcritBC} can be rewritten as
$$\alpha^2 p(L)p'(L) + 2\lambda \coth{(\lambda L)} \, p'(L)^2 \ge 0,$$
which must be true for all $\lambda = \sqrt{\alpha^2+n^2}$. Since the minimum value of
$\lambda \coth{(\lambda L)}$ is attained when $n=1$, we obtain \eqref{W1storderbest}.
\end{proof}

Now we solve the differential inequality \eqref{W2ndorderbest} subject to the condition \eqref{W1storderbest}. Note that
whether $p$ is positive or negative, the differential inequalities remain the same expressed in terms of $\lvert p\rvert$.

\begin{proposition}
Let $\lvert p(y)\rvert = Z(y)^{-\alpha/(2\alpha^2+2)}$ for some function $Z(y)$. Then the differential inequality \eqref{W2ndorderbest}
becomes $Z''(y) \le  \lambda^2 Z(y)$, while the initial condition \eqref{W1storderbest} becomes
\begin{equation}\label{ZBC}
 -\lambda \sinh{\lambda L} \, Z(L) Z'(L) + \cosh{\lambda L}\, Z'(L)^2\ge 0
\end{equation}
for $\lambda^2 = \alpha^2+1$.
\end{proposition}

In particular the ``critical'' function $Z(y)$ is given by
$$ Z(y) = z_0 \sinh{\big(\lambda (y-y_0)\big)} $$
for any constants $y_0$ and $z_0$: in this case we get $R(y) = -\dfrac{\cosh{\big(\lambda(L-y_0)\big)}}{\cosh{(\lambda y_0)}}$ which is obviously negative.
This translates into
$$ \alpha^2 \psi'(y) - \beta = Z(y)^{-\tfrac{\alpha^2}{2\lambda^2}}$$
which can be integrated to find $\psi(y)$. Meanwhile condition \eqref{ZBC} becomes
$\cosh{(\lambda y_0)} \ge 0,$ 
which is obviously true.

\section{Concluding remark and future research}

In the case that $\alpha=0$, every function $\psi(y)$ will satisfy the criterion of Corollary \ref{nonnegativecurvaturecor} regardless of $\beta$. On the other hand, for nonzero $\alpha$ the curvature can become positive. Hence we can view the Froude number $\alpha^2$ as stabilizing.

In future research we can perform the same computations in more general rotationally symmetric geometry, e.g., on the $2$-sphere, as in \cite{P2005}. In general, similar techniques should yield relatively simple curvature formulas for incompressible fluids in higher dimensions or with symmetry. For example a similar approach yields curvature results for standard axisymmetric fluids in \cite{WP2017}, and we can try the same for flows with symmetry in more general Riemannian $3$-manifolds (e.g., the Thurston geometries).

\vspace{0.25in}
\begin{acknowledgement}
J. Lee was supported by Gustafsson Foundation, Sweden and S. C. Preston was supported by Simons Foundation, Collaboration Grant for Mathematicians, no. 318969.
\end{acknowledgement}

%
%
%


\vfill

\begin{thebibliography}{10}

\bibitem{A1966}
Arnold, V.I., 
\newblock{On the differential geometry of infinite-dimensional Lie groups and its application to the hydrodynamics of perfect fluids},
\newblock translation of the French original, in {\em Vladimir I. Arnold: collected works vol. 2}, Springer, New York, 2014.


\bibitem{AK1999}
V.~Arnold~and~B.~Khesin,
\newblock{Topological methods in hydrodynamics},
\newblock{\em Springer}, Vol. 125, 1999

\bibitem{BHP2019}
M.~Bauer, P.~Harms, and S.C.~Preston,
\newblock{Vanishing distance phenomena and the geometric approach to SQG},
\newblock{arXiv:1805.04401}

\bibitem{EM1970}
D.~Ebin~and~J.~Marsden,
\newblock{Groups of diffeomorphisms and the motion of an incompressible fluid},
\newblock{\em Ann. Math.}, 102--163, 1970

\bibitem{EP2015}
D.~Ebin~and~S.~Preston,
\newblock{Riemannian geometry of the contactomorphism group},
\newblock{\em Arnold Math. J.}, 1(1), 5--36, 2015

\bibitem{HZ1998}
D.~Holm~and~V.~Zeitlin,
\newblock{Hamilton’s principle for quasigeostrophic motion},
\newblock{\em Phys. Fluids}, 10(4), 800--806, 1998

\bibitem{KLMP2013}
B.~Khesin,~J.~Lenells,~G.~Misiolek,~and~S.~Preston,
\newblock{Curvatures of Sobolev Metrics on Diffeomorphism Groups}
\newblock{\em Pure Appl. Math. Quart.}, 9(2), 291--332, 2013

\bibitem{M2003}
A.~Majda,
\newblock{Introduction to PDEs and Waves for the Atmosphere and Ocean},
\newblock{\em Amer. Math. Soc.}, Vol. 9, 2003

\bibitem{M1993}
G.~Misiolek,
\newblock{Stability of Flows of Ideal Fluids and the Geometry of the Group of Diffeomorphisms},
\newblock{\em Indiana Univ. Math. J.}, 42(1), 215--235, 1993

\bibitem{Ped2013}
J.~Pedlosky,
\newblock{Geophysical fluid dynamics},
\newblock{\em Springer}, 2013.

\bibitem{P2004}
S.C.~Preston,
\newblock{For ideal fluids, Eulerian and Lagrangian instabilities are equivalent},
\newblock{\em Geom. Funct. Anal.} 14(5), 1044--1062, 2004

\bibitem{P2005}
S.C.~Preston,
\newblock{Nonpositive curvature on the area-preserving diffeomorphism group},
\newblock{\em J. Geom. Phys.}, 53(2), 226--248, 2005

\bibitem{V2008}
C.~Vizman,
\newblock{Cocycles and stream functions in quasigeostrophic motion},
\newblock{\em J. Nonlin. Math. Phys.}, 15(2), 140--146, 2008

\bibitem{W2016}
P.~Washabaugh,
\newblock{The SQG equation as a geodesic equation},
\newblock{\em Arch. Rat. Mech. Anal.}, 222(3), 1269--1284, 2016

\bibitem{WP2017}
P.~Washabaugh and S.C.~Preston,
\newblock{The geometry of axisymmetric ideal fluid flows with swirl},
\newblock{\em Arnold Math. J.} 3(2), 175–-185, 2017

\bibitem{ZP1994}
V.~Zeitlin~and~R.A.~Pasmanter,
\newblock{On the differential geometry approach to geophysical flows},
\newblock{\em Phys. Lett. A}, 189, 59--63, 1994


\end{thebibliography}
\end{document}